\documentclass{amsart}
\usepackage{amsfonts}
\usepackage{amsmath,amscd}
\usepackage{amssymb}
\usepackage{amsthm}
\usepackage{newlfont}

\long\def\alert#1{\parindent2em\smallskip\hbox to\hsize
{\hskip\parindent\vrule
\vbox{\advance\hsize-2\parindent\hrule\smallskip\parindent.4\parindent
\narrower\noindent#1\smallskip\hrule}\vrule\hfill}\smallskip\parindent0pt}

 \newtheorem{thm}{Theorem}[section]
 \newtheorem{cor}[thm]{Corollary}
 \newtheorem{lem}[thm]{Lemma}
 
\theoremstyle{definition}
 \newtheorem{defn}[thm]{Definition}
\theoremstyle{remark}

 \numberwithin{equation}{section}
 \newtheorem*{theorem A}{\textbf{ Theorem  A}}
\newtheorem*{theorem  B}{\textbf{ Theorem   B}}
\newtheorem*{corollary  C}{\textbf{ Corollary   C}}
\newtheorem*{corollary D}{\textbf{ Corollary  D}}
\newtheorem*{theorem E}{\textbf{ Theorem  E}}
\newtheorem*{theorem F}{\textbf{ Theorem  F}}
\newtheorem*{theorem G}{\textbf{ Theorem  G}}
\newtheorem*{theorem H}{\textbf{ Theorem  H}}
\newtheorem*{corollary K}{\textbf{ Corollary  K}}
\newtheorem*{p a}{\textbf{Proof of Theorem  A}}
\newtheorem*{p b}{\textbf{Proof of Theorem  B}}
\newtheorem*{p c}{\textbf{Proof of Theorem  C}}
\newtheorem*{p d}{\textbf{Proof of Corollary  D}}
\newtheorem*{p e}{\textbf{Proof of Theorem  E}}
\newtheorem*{p f}{\textbf{Proof of Theorem  F}}
\newtheorem*{p g}{\textbf{Proof of Theorem  G}}
\newtheorem*{p h}{\textbf{Proof of Theorem  H}}

\begin{document}
%
%
%
\title[$2$-capability  of $2$-generator   $p$-groups of class two ]
 {  $2$-nilpotent multiplier and   $2$-capability  of finite $2$-generator $p$-groups of class two  }

 \author{F. Johari}
\author{A. Kaheni $ ^{*} $ }\thanks{${}^{*}{Corresponding ~author.~ Email: azamkaheni@birjand.ac.ir}$}

\address{Departamento de Matem\'atica, Instituto de Ci\^encias Exatas, Universidade  Federal de Minas Gerais, Av. Ant\^onio Carlos 6627, Belo Horizonte, MG, Brazil}
\email{farangisjohari@ufmg.br, farangisjohari85@gmail.com} 

\address{Department of Mathematics\\
 University of Birjand, Birjand, Iran}
\email{azamkaheni@birjand.ac.ir}


\maketitle
\begin{abstract}
Let $p$ be a prime number. We give the explicit structure of   $2$-nilpotent multiplier for each finite $ 2$-generator $ p$-group of class two. Moreover,  $2$-capable groups in that class are characterized.
\end{abstract}

\noindent

\vspace{.5cm}\leftline{Keywords:  $2$-Capable group, $2$-Nilpotent
multiplier, $ p$-Group.}

\section{Introduction and  Motivation}
The first classification for finite $2$-generator $p$-groups of  class two for an odd prime $ p $ was  given by  Bacon and Kappe   \cite{BK}. In $1999,$ Kappe, Vissher, and Sarmin  \cite{KSV} generalized the previous classification to  the case $p=2.$  Later, Ahmad, Magidin, and Morse  \cite{AMM} discovered that these classifications were incomplete. They corrected these omissions by giving  a full classification of these groups as stated in the following.
\begin{thm}\cite[Theorem 1.1]{AMM}\label{pres}
 Let $p$ be a prime and $n > 2. $  Every $2$-generated $p$-group of class exactly $2$ corresponds to an ordered $5$-tuple of integers, $ (\alpha,\beta,\gamma;\rho,\sigma),$ such that: 
 \item[$(i)$]  $ \alpha\geq\beta\geq\gamma\geq 1;$
\item[$(ii)$]  $ \alpha+\beta+\gamma =n;$ 
\item[$(iii)$]  $0\leq\rho\leq\gamma$ and $0\leq \sigma\leq \gamma;$
where $ ( \alpha, \beta, \gamma; \rho, \sigma) $ corresponds to the group presented by
\[G=\langle a, b \ \vert \ [ a , b ] ^{ p^{ \gamma }} =[ a, b, b ]= [ a, b , a] = 1,\ a ^{ p^{ \alpha }} =[ a , b ] ^{ p^{ \rho }},  b ^{p^{ \beta } } =[ a , b ] ^{ p^{ \sigma }} \rangle. \]
Moreover:\newline
$\mathrm{FAMILY} $ $(1)$   If  $\alpha>\beta,$  then $G $ is isomorphic to:
\item[$(a)$]$    (\alpha, \beta, \gamma; \rho,\gamma),$ when $\rho\leq \sigma.$ 
\item[$(b) $]$   (\alpha, \beta, \gamma; \gamma, \sigma),$ when $0 \leq\sigma <\sigma+\alpha-\beta\leq \rho$ or $\sigma<\rho=\gamma.$
\item[$(c)$]$  (\alpha, \beta, \gamma; \rho, \sigma),$  when  $0 \leq\sigma < \rho< \mathrm{min} (\gamma, \sigma+\alpha-\beta).$\newline
$\mathrm{FAMILY} $  $(2) $  If $ \alpha=\beta>\gamma,$ or  $ \alpha=\beta=\gamma,$ and $p > 2,$ then $G$ is isomorphic to $ (\alpha, \beta, \gamma; \mathrm{min}(\rho,\sigma),\gamma).$\newline
$\mathrm{FAMILY }$ $(3)$ If   $ \alpha=\beta=\gamma,$ and $p = 2,$ then $G$ is isomorphic to:
\item[$ (a) $ ] $ (\alpha, \beta, \gamma; \mathrm{min}(\rho,\sigma),\gamma),$ when $0\leq \mathrm{min}(\rho,\sigma)<\gamma-1.$
\item[$(b) $ ] $ (\alpha, \beta, \gamma; \gamma-1, \gamma-1),$ when $\rho=\sigma=\gamma-1.$
\item[$(c)$]$ (\alpha,\beta,\gamma;\gamma,\gamma),$ when $ \mathrm{min}(\rho,\sigma)\geq\gamma-1$ and $\mathrm{max}(\rho,\sigma)=\gamma.$ \\
     The groups listed in $1(a)$-$3(c)$ are pairwise nonisomorphic. 
\end{thm}
 The idea of computing the nonabelian tensor square and the nonabelian exterior square for these groups was started in \cite{BK,KSV}, but these computations were based on the previous classification. Magidin and Morse \cite{MM} computed various homological functors for these groups, using the new classification in \cite{AMM}.  These functors include the nonabelian tensor square, the nonabelian exterior square, and the Schur multiplier. They also determined which of these groups are capable. A group $ G $ is called capable if $ G $ is isomorphic to the center factor of some group $ H. $ The following result gives necessary and sufficient conditions for the capability of  $2$-generator $p$-groups of class two.  
\begin{lem}\cite[Theorems 63 and 67]{MM}\label{mcap}
 Let $G = G_p(\alpha,\beta,\gamma;\rho,\sigma)$  be a $2$-generator $p$-group of class two with the presentation as in Theorem \ref{pres}.  Then the following results hold:
 \begin{itemize}
 \item[$(a)$] For an odd prime number $p,$  $G$ is capable if and only if $\alpha-\beta=\rho-\sigma$ and $\beta=\rho.$
\item[$(b)$] For $p=2,$ $G$ is capable if and only if it meets one of the following conditions:
\begin{itemize}
\item[$(i)$] $\rho\leq\sigma, \alpha=\beta,$ and $\rho=\gamma;$
\item[$(ii)$]$\rho>\gamma, (\alpha-\beta)=(\rho-\sigma)>\delta_{\beta\gamma},$ and $\rho=\gamma,$ where $\delta_{\beta\gamma}$ is the Kronecker delta; 
\item[$(iii)$]$\rho=\sigma=\gamma=\beta, $ and $ \alpha=\gamma+1.$
\end{itemize}
 
  \end{itemize}
\end{lem}
Using Theorem \ref{pres} and Lemma \ref{mcap}, all capable $2$-generator $p$-groups of class two are listed as follows:
\begin{cor}\label{2cap1}
Let $G$ be a  $2$-generator $p$-group of class two. Then the following assertions hold:
 \begin{itemize}
 \item[$(a)$] If  $p$ is odd, then $G$ is capable if and only if $ G $ is  isomorphic to exactly one of the following groups.

 \begin{itemize}
 \item[$(i)$] $G\cong G_1=\langle a, b \ \vert \ [ a , b ] ^{ p^{ \alpha }} = a ^{ p^{ \alpha }} = b ^{p^{ \alpha } } = [ a, b, b ]= [ a, b , a] = 1\rangle. $ 
 \item[$(ii)$]$G\cong G_2=\langle a, b \ \vert \ [ a , b ] ^{ p^{ \gamma}} = a ^{ p^{ \alpha }} = b ^{p^{ \alpha } } = [ a, b, b ]= [ a, b , a] = 1,~\alpha > \gamma\rangle.$ 
 \item[$(iii)$] $G\cong G_3= \langle a, b \ \vert \ [ a , b ] ^{ p^{ \gamma}} = a ^{ p^{ \alpha }} = [ a, b, b ]= [ a, b , a] = 1,  b ^{p^{ \beta } } =[ a , b ] ^{ p^{ \sigma}}  \rangle,$  where $\alpha-\beta=\gamma-\sigma$ and $ 0\leq  \sigma <\gamma.$
\end{itemize}
\item[$(b)$] If  $p=2,$  then  $G$ is capable if and only if $G$ is isomorphic to exactly one of the following groups.
\begin{itemize}
\item[$(iv)$]  $G\cong G_4=\langle a, b \ \vert \ [ a , b ] ^{ 2^{ \beta }} = a ^{ 2^{ \beta }} = b ^{2^{ \beta } } = [ a, b, b ]= [ a, b , a] = 1\rangle.$
\item[$(v)$]$G\cong G_5= \langle a, b \ \vert \  [ a , b ] ^{ 2^{ \beta }} = a ^{ 2^{ \beta + 1}} = b ^{2^{ \beta } } = [ a, b, b ]= [ a, b , a] = 1\rangle.$
\item[$(vi)$]$G\cong G_6=\langle a, b \ \vert \ [ a , b ] ^{ 2^{ \gamma}} = a ^{ 2^{ \alpha }} = b ^{2^{ \alpha } } = [ a, b, b ]= [ a, b , a] = 1,\alpha > \gamma\rangle.$
\item[$(vii)$]$G\cong G_7=\langle a, b \ \vert \ [ a , b ] ^{ 2^{ \gamma}} = a ^{ 2^{ \alpha }} = [ a, b, b ]= [ a, b , a] = 1,  b ^{2^{ \beta } } =[ a , b ] ^{ 2^{ \sigma}}  \rangle,$
 where $ \alpha - \beta > \delta_{  \beta \gamma}, $  $ \alpha - \beta = \gamma - \sigma,$  $ \alpha > \beta,$  $ \gamma >\sigma,$  and  $\delta_{ij}$ is the Kronecker  delta.
\end{itemize}
\end{itemize}

\end{cor}
 
Burns and Ellis  \cite{BE} generalized the notion of   capability for groups to  $c$-capability.   A group $G$ is called $c$-capable if  $G\cong H/Z_c(H)$ for some group $H,$ where $ Z_c(H)  $ is the $c$-th term of the upper central series of $ H $ for all $ c $ with $ c\geq 1. $ Clearly, each $(c+1)$-capable group is $c$-capable. The converse of this statement is not true in general. A counterexample was given by Burns and Ellis  \cite{BE}. Despite of this flaw example, someone try to find  groups in which two concepts of  capability and  $c$-capability are equivalent for them. Finitely generated abelian groups, extra special $p$-groups, and also generalized extra special  $p$-groups are some such groups (for more details, see \cite{BE,NP, NJP}). Recently,  Monfared, Kayvanfar, and Johari \cite{FKJ} showed that  “capability” and “$2$-capability” are coincide for finite $2$-generator  $2$-groups of class two.
More precisely, they gave a classification of all $2$-capable  $2$-generator $2$-groups of class two and proved that for these groups each capable group is  $2$-capable.  Furthermore, they computed  the $2$-nilpotent multiplier of  $2$-capable  $2$-generator $2$-groups of class two. \\
The main result of the present paper is to classify all $2$-capable finite $2$-generator $ p$-groups of class two. At first, we determine the structure of $2$-nilpotent multipliers for all capable $2$-generator $p$-groups of class two when $ p$  is odd. Then for these  $p$-groups, we show that  each capable group is $2$-capable. Finally, we compute  $2$-nilpotent multipliers for all noncapable $2$-generator $p$-groups of class two for an arbitrary prime number $ p,$ to clarify  the structure of  $2$-nilpotent multiplier for each $2$-generator $ p$-group of class two.\par
Now, we state our main results.\\ The following result describes  $2 $-nilpotent multipliers for all capable finite $ 2 $-generator $ p $-groups of class two for $ p>2.$ 
\begin{theorem A}\label{k3}
{\em
Let $ p$ be an odd prime number and $G$  be a capable finite $ 2 $-generator $ p $-group of class two  with  the presentation as in Corollary  \ref{2cap1}(a).  Then  \begin{align*}\label{eq6}
 \mathcal{M}^{(2)}(G) \cong \left\{ \begin{array}{rl} 
  \mathbb{Z} _{ p^{ \alpha}} ^{ (5) }  \qquad\qquad    \qquad \qquad &\text{if}~G\cong G_1, \\
                                    \mathbb{Z}_{ p^{ \alpha}}^{ (2)}\oplus \mathbb{Z}_{ p^{ \gamma}}^{ (3) } \qquad \quad \qquad   &\text{if}~G\cong G_2,\\\mathbb{Z} _{ p^{ \alpha}  }\oplus\mathbb{Z}_{ p^{ \beta}}\oplus \mathbb{Z}_{ p^{ \sigma} }^{ (3) } \qquad \   &\text{if}~G\cong G_3, \\
                                \end{array} \right.                                
\end{align*}
where $\mathbb{Z} _{ r} ^{ (t) }$ denotes the direct sum of $t$ copies of $ \mathbb{Z} _r,  $ in which  $ \mathbb{Z} _r  $ is the cyclic group of order $r.$
}
\end{theorem A}
The next result gives the exact structure of  $2$-capable finite $ 2 $-generator $ p $-groups of class two. 
\begin{theorem B}
{\em Let $G$ be a  finite $2$-generator $p$-group of class two for odd prime $ p. $ Then $G$ is $2$-capable if and only if $ G $ is isomorphic to exactly one of the $p$-groups $ G_1, $ $ G_2,$ or $ G_3.$}
\end{theorem B}
 Now, one can obtain a full classification of all  $2$-capable finite $ 2 $-generator $ p $-groups of class two, using Theorem B and  \cite[Theorem 4.3]{FKJ}.
\begin{corollary C}\label{2cap}
{\em  Let $G$ be a  finite $2$-generator $p$-group of class two. Then the following results hold:
\begin{itemize}
\item[$ (a) $] If  $p$ is an odd prime number, 
 then $G$ is $2$-capable if and    only if $G$  is isomorphic to exactly one of the following $p$-groups.\\
$\qquad (i)$ $G\cong G_1=\langle a, b \ \vert \ [ a , b ] ^{ p^{ \alpha }} = a ^{ p^{ \alpha }} = b ^{p^{ \alpha } } = [ a, b, b ]= [ a, b , a] = 1\rangle.$\\
$\ \ (ii)$ $G\cong G_2=\langle a, b \ \vert \ [ a , b ] ^{ p^{ \gamma}} = a ^{ p^{ \alpha }} = b ^{p^{ \alpha } } = [ a, b, b ]= [ a, b , a] = 1,\alpha > \gamma\rangle.$\\
$\ \ (iii)$ $G\cong G_3= \langle a, b \ \vert \ [ a , b ] ^{ p^{ \gamma}} = a ^{ p^{ \alpha }} = [ a, b, b ]= [ a, b , a] = 1,  b ^{p^{ \beta } } =[ a , b ] ^{ p^{ \sigma}}  \rangle,$ where $\alpha-\beta=\gamma-\sigma$ and $ 0\leq \sigma <\gamma.$
\item[$ (b) $] If  $p=2,$  then  $G$ is $2$-capable if and only if $G$   is isomorphic to exactly one of the following 
$2$-groups.\\
$\ \ (iv)$  $G\cong G_4=\langle a, b \ \vert \ [ a , b ] ^{ 2^{ \beta }} = a ^{ 2^{ \beta }} = b ^{2^{ \beta } } = [ a, b, b ]= [ a, b , a] = 1\rangle.$\\
$\ \ (v)$ $G\cong G_5= \langle a, b \ \vert \  [ a , b ] ^{ 2^{ \beta }} = a ^{ 2^{ \beta + 1}} = b ^{2^{ \beta } } = [ a, b, b ]= [ a, b , a] = 1\rangle.$\\
$\ \ (vi)$ $G\cong G_6=\langle a, b \ \vert \ [ a , b ] ^{ 2^{ \gamma}} = a ^{ 2^{ \alpha }} = b ^{2^{ \alpha } } = [ a, b, b ]= [ a, b , a] = 1,\alpha > \gamma\rangle.$\\
$\ \ (vii)$ $G\cong G_7=\langle a, b \ \vert \ [ a , b ] ^{ 2^{ \gamma}} = a ^{ 2^{ \alpha }} = [ a, b, b ]= [ a, b , a] = 1,  b ^{2^{ \beta } } =[ a , b ] ^{ 2^{ \sigma}}  \rangle,$
 where $ \alpha - \beta > \delta_{  \beta \gamma}, $  $ \alpha - \beta = \gamma - \sigma,$  $ \alpha > \beta,$  $ \gamma >\sigma,$  and  $\delta_{ij}$ is the  Kronecker delta.
\end{itemize}
}
\end{corollary C}
The following result is immediately obtained from  Corollary  \ref{2cap1}, Theorem B,  and \cite[Corollary 4.4]{FKJ}.
\begin{corollary D}\label{2cap}{\em
Let $G$ be a  finite $2$-generator $p$-group of class two. Then $G$ is $2$-capable if and only if  $G$ is capable.}
\end{corollary D}
Recall that the epicenter of a group $ G, $ $ Z^*(G),$ is the smallest central subgroup of $G$ with capable quotient, which  is defined in \cite{bey}. In particular, $ G $ is capable if and only if $Z^*(G)=1.  $ Following \cite{BE,MK}, the  $2$-epicenter of $ G,$  $ Z_2^*(G),$ is the smallest subgroup of $Z_2(G)$ with $2$-capable quotient. In fact, $ G $ is $2$-capable if and only if  $ Z_2^*(G)$ is trivial. Moreover, $ Z^*(G)\subseteq Z_2^*(G).$ (For more information, see \cite[Lemma 2.1]{BE} ).\\
Here, we prove that the epicenter and the $2$-epicenter for  a finite  $ 2 $-generator $ p $-group of class two are coincide.
\begin{theorem E}
{\em Let $ G $ be a finite  $ 2 $-generator $ p $-group of class two. Then   $ Z^*(G)=Z_2^*(G). $}
  \end{theorem E}
In the sequel, we give the exact structure of $2$-nilpotent multipliers for all capable $2$-generator $p$-groups of class two  in Sections $2$ and $4.$ Moreover, the $2$-capability of such groups is discussed in Section $3.$

\section{$2 $-Nilpotent multipliers of some capable $ p $-groups }
Let $  G$ be a group presented as $ F/R $ to be the quotient group of a free group $F$ by a normal subgroup  $R.$ From \cite{LM}, the Baer invariant of a group $G$  with respect to the variety of nilpotent groups of class at most $2$  is called the $2$-nilpotent multiplier of $G,$  $\mathcal{ M}^{ (2) } (G),$ and  defined as follows:
\begin{center}
$\mathcal{ M}^{ (2) } (G)\cong \big{(}R \cap \gamma _{ 3} (F)\big{)} / [R, F, F],$
\end{center}
in which $ \gamma _{3} (F)  = [F, F, F].$ It is well known that 
the $2$-nilpotent multiplier of $G$ is abelian and independent of the choice of its free presentation (see \cite{LM}).

We need  the notion of  basic commutator for computing  the $2$-nilpotent multipliers of groups.
  \begin{defn}
Let $ X $ be an arbitrary subset of a free group and select an arbitrary total order
for $ X. $ The basic commutators on $ X, $ their weight $ wt, $ and the ordering among them
are defined as follows:
\begin{itemize}
\item[$(i)$] The elements of $ X $ are basic commutators of weight one, ordered according to the total order previously chosen. 
\item[$(ii)$] Having defined the basic commutators of weight less than $ n, $ a basic commutator of weight $  n $ is $d = [s,k ],$ where:
\begin{itemize}
\item[$(a)$] $ s $ and $ k $ are basic commutators and $wt(s) + wt(k) = n,$ and
\item[$(b)$] $s > k,$  and if $s = [s_{1} , s_{2} ],$ then $ k \geq s_{2} .$
\end{itemize}
\item[$(iii)$] The basic commutators of weight $ n $ follow those of weight less than $ n. $ The basic commutators of weight $n$ are ordered among themselves in any total order, but the most common used total order is lexicographic order, that is if $[b_{1}, a_{1}] $ and $[b_{2}, a_{2}] $ are basic commutators of weight $ n, $ then $[b_{1} , a_{1} ] < [b_{2 }, a_{2} ]$ if and only if $b_1 < b_2$ or $b_1= b_2$ and $a_1 < a_2.$ 
\end{itemize}
\end{defn}
Let $G$ be a finite $2$-generator $p$-group of class two for $ p>2$ with the free presentation $ F/R $ such that $ F $ is a free group generated by $ \{a,b\} $ with a normal subgroup $R.$ Then $\gamma_3(F)\subseteq R,$ because of nilpotency of $G.$ Thus
\begin{center}
$\mathcal{ M}^{ (2) } (G)\cong \dfrac{R \cap \gamma _{ 3} (F) }{[R,F,F]}\cong\dfrac{  \gamma _{ 3} (F)/\gamma_5(F)}{[R,F,F] /\gamma_5(F)},$
\end{center}
where $\gamma_5(F)  $ is the $5$-th term of the lower central series of $ F.$
It is known that $\gamma_3(F)/\gamma_5(F)$ is the free abelian group with the basis of all basic commutators of weights $3$ and $4$ on $\{a,b\}.$ By considering  $ a >b,$ we have
\[\gamma_3(F)/\gamma_5(F)= \langle [a,b,a]\gamma_5(F),[a,b,b]\gamma_5(F),[a,b,b,b]\gamma_5(F),\]\[[a,b,a,a]\gamma_5(F),[a,b,b,a]\gamma_5(F)\rangle.\]
Therefore, one can gain the structure of $\mathcal{ M}^{ (2) } (G)$ by providing a suitable  basis  for $[R,F,F]/\gamma_5(F).$
\par 
Now, we determine  the  structure of  $2$-nilpotent multipliers of  $p$-groups in Corollary \ref{2cap1}, whenever $p$ is odd.

\begin{lem}\label{k3}
Let $ G $ be the group $G_1$  presented  in Corollary \ref{2cap1}(a)$(i).$
 Then \[ \mathcal{M}^{(2)}(G) \cong \mathbb{Z} _{ p^{ \alpha}} ^{ (5) }. \]
\end{lem}
\begin{proof} 
Let $F/R$  be a free presentation for $G_1$ such that  $F$ is a free group on $\{ a, b\}$ and $R=\langle   a ^{ p^{ \alpha }}, b ^{p^{ \alpha }}, [ a , b ] ^{p^{\alpha}}, [ a, b, b ], [ a, b , a] \rangle^F.$ Hence, 
\[\dfrac{[ R , F , F ]}{\gamma _{5 } ( F) }=\dfrac{ \langle [ a^{ p^{ \alpha } } , f_1 , f_2 ],  [ b^{ p^{ \alpha } }, f_3 , f_4 ] , 
[ [a, b]^{ p^{ \alpha } }, f_5 , f_6 ]  \mid  f_i \in F, 1\leq i \leq 6 \rangle^{F}  \gamma _{5 } ( F) }{\gamma _{5 } ( F) }.
\]
Using 
\cite[Lemma 3.3]{NP}, we obtain
\[[ a^{p^{\alpha}} , f_1 , f_2 ]^f \equiv [ a , f_1 , f_2 ]^{p^{\alpha}}[ a , f_1,a , f_2 ]^{\left(^{p^{\alpha}}_{2} \right)}[ a , f_1 , f_2,f ]^{p^{\alpha}}(\mod  \gamma_{5} ( F) ), 
\]
\[[ b^{p^{\alpha}} , f_3 , f_4 ]^f \equiv [ b , f_3 , f_4 ]^{p^{\alpha}}[ b , f_3,b , f_4 ]^{\left(^{p^{\alpha}}_{2} \right)}[ b, f_3 , f_4,f ]^{p^{\alpha}}(\mod  \gamma_{5} ( F) ) , 
\]
\[ \text{and}~[ [a, b]^{p^{\alpha}}, f_5 , f_6 ]^f \equiv [ a, b, f_5 , f_6 ]^{p^{\alpha}} \  \  (\mod \gamma_{5} ( F) )  \]
for all $ f_i,f\in F. $
Since $p>2,$  $[ a , f_1 , f_2 ]^{ p^{ \alpha } }\gamma _{5 } ( F)\in [ R , F , F ]/ \gamma _{5 } ( F).$ Therefore,
\[A= \{[ a , b , b ]^{ p^{ \alpha } } \gamma _{5 } ( F),  [ a, b , a]^{ p^{ \alpha } } \gamma _{5 } ( F), [ a , b , b,b ]^{ p^{ \alpha } } \gamma _{5 } ( F), \]\[ [ a , b , b,a ]^{ p^{ \alpha } } \gamma _{5 } ( F), [ a , b , a,a ]^{ p^{ \alpha } } \gamma _{5 } ( F) \}\] generates $ [ R , F , F ] /  \gamma _{5 } ( F). $ Since  $ \gamma_3(F) /  \gamma _{5 } ( F) $ is a free abelain group, it is easy to see that  $ A $ is linearly independent, and so $ A $ is a basis for $ [ R , F , F ] /  \gamma _{5 } ( F).$ Thus
\begin{equation*}
\begin{split}
[ R , F , F ] \equiv
& \langle [ a , b , b ]^{ p^{ \alpha } },  [ a, b , a]^{ p^{ \alpha } } , [ a , b , b,b ]^{ p^{ \alpha } },[ a , b , b,a ]^{ p^{ \alpha } }, [ a , b , a,a ]^{ p^{ \alpha } } \rangle
  ( \mod  \gamma _{5 } ( F) ).
\end{split}
\end{equation*}
Since $\gamma _{3 } ( F)/\gamma_{5}( F)\cong \mathbb{Z}^{ (5)} $ and   $[ R , F , F ]/\gamma_{5}( F)\cong  ({ p^{ \alpha}}\mathbb{Z})^5,$ we get  $ \mathcal{M}^{(2)}(G_1) \cong \mathbb{Z} _{ p^{ \alpha}} ^{ (5) }. $ This proof is complete.
\end{proof}
\begin{lem}\label{k2}
Let $G $ be the group $ G_2$ presented in Corollary \ref{2cap1}(a)$(ii).$ Then \[ \mathcal{M}^{(2)}(G) \cong \mathbb{Z}_{ p^{ \alpha}}^{ (2)}\oplus \mathbb{Z}_{ p^{ \gamma}}^{ (3) }. \]
\end{lem}
\begin{proof}
Let $F/R$  be a free presentation for $G_2$ such that  $F$ is a free group on $\{ a, b\}$ and $R=\langle[ a , b ] ^{ p^{ \gamma}} , a ^{ p^{ \alpha }} , b ^{p^{ \alpha } }, [ a, b, b ],  [ a, b , a], \alpha  > \gamma\rangle^F.  $ By some commutator computations similarly to the proof of Lemma \ref{k3}, one can reach that
\begin{equation*}
\begin{split}
[ R , F , F ] \equiv
& \langle [ a , b , b ]^{ p^{ \alpha } },  [ a, b , a]^{ p^{ \alpha } } , [ a , b , b,b ]^{ p^{ \gamma } },[ a , b , b,a ]^{ p^{ \gamma } }, [ a , b , a, a ]^{ p^{ \gamma } } \rangle
  ( \mod
   \gamma_{5 } ( F) ).
\end{split}
\end{equation*}
Thus  $ \mathcal{M}^{(2)}(G_2) \cong \dfrac{ \gamma _{3 } ( F)/\gamma_{5}( F) }{[ R , F , F ]/\gamma_{5 } ( F) }\cong  \mathbb{Z}_{ p^{ \alpha}}^{ (2)}\oplus \mathbb{Z}_{ p^{ \gamma} }^{ (3) },$ as desired.

\end{proof}
\begin{lem}\label{k1}
Let $G $ be the group $G_3$ presented  in  Corollary \ref{2cap1}(a)$(iii).$
   Then \[
 \mathcal{M}^{(2)}(G) \cong
                                    \mathbb{Z} _{ p^{ \alpha}  }\oplus\mathbb{Z}_{ p^{ \beta}}\oplus \mathbb{Z}_{ p^{ \sigma} }^{ (3) }.\]
\end{lem}
\begin{proof}
Let $F/R$  be a free presentation for $G_3$ such that  $F$ is a free group on $\{ a, b\}$ and $R=\langle[ a , b ] ^{ p^{ \gamma}} , a ^{ p^{ \alpha }} , [ a, b, b ], [ a, b , a] ,  b ^{p^{ \beta } } [ a , b ] ^{ p^{ -\sigma}}  \rangle^F.$  
One can easily check that $[ R , F , F ]$ is generated by the following set
\[\{[ a , b , f_1,f_2 ]^{ p^{ \gamma } }\gamma _{5 } ( F) , [a,f_3,f_4]^{ p^{ \alpha }}\gamma _{5 } ( F), [b ,f_5,f_6]^{p^{ \beta } }[ a , b  , f_5,f_6]^{ p^{ -\sigma}}\gamma _{5 } ( F) \]\[ \mid f_i\in F,1\leq i\leq 6\}.\] Similarly to the proof of Lemma \ref{k3},
  $ [ R , F , F ]/\gamma_{5 } ( F)$ is generated by the elements $ [a,b,a]^{ p^{ \alpha }}\gamma_{5 } ( F), $ $ [a,b,b]^{ p^{ \beta }}\gamma_{5 } ( F),$ $[a ,b,b]^{p^{- \beta } } [ a , b  ,a ,a]^{ p^{ -\sigma}}\gamma_{5 } ( F), \  [ a, b, b,b]^{ p^{ -\sigma}}\gamma_{5 } ( F),$  and $[ a, b , b,a]^{ p^{ -\sigma}}\gamma_{5 } ( F).$  
Therefore,  the following set \[ \{[ a, b, b ]^{ p^{ \beta } },  [ a, b , a]^{ p^{ \alpha } } , [ a, b, b,b ]^{ p^{ \sigma } },[ a, b, b, a ]^{ p^{ \sigma } }, [ a, b, b ]^{ p^{ \beta } } [ a, b, a,a ]^{ p^{ \sigma } }\} \]
is a basis for the group $[ R , F , F ]$ in modulo $\gamma _{5 } ( F). $
By easy  calculations,    the set $ \{[ a, b, b ],  [ a, b , a] , [ a, b, b,b ],[ a, b, b, a ], [ a, b, b ]^{p^{\beta-\sigma}}[ a, b, a,a ]\}$ is also a basis for the free abelian group $\gamma_3 (F)/\gamma_5(F).$ Thus
\begin{align*}
\mathcal{M}^{(2)}(G_3) &\cong \dfrac{ \gamma _{3 } ( F)/\gamma_{5}( F) }{[ R , F , F ]/\gamma_{5 } ( F) }\cong  \dfrac{[ a, b, b ]}{[ a, b, b ]^{ p^{ \beta } }}\oplus \dfrac{[ a, b , a]}{[ a, b , a]^{ p^{ \alpha } }}
\oplus \dfrac{[ a, b, b,b ]}{[ a, b, b,b ]^{ p^{ \sigma } }}\\&\oplus \dfrac{[ a, b, b, a ]}{[ a, b, b, a ]^{ p^{ \sigma } }}\oplus \dfrac{ [ a, b, b ]^{ p^{ \beta - \sigma} } [ a, b, a,a ]}{([ a, b, b ]^{ p^{ \beta - \sigma} } [ a, b, a,a ])^{ p^{ \sigma } }}\cong 
                                    \mathbb{Z} _{ p^{ \alpha}  }\oplus\mathbb{Z}_{ p^{ \beta}}\oplus \mathbb{Z}_{ p^{ \sigma} }^{ (3) } ,
\end{align*}
 as required.
\end{proof}
Now, we are ready to prove Theorem A.
\begin{p a} The result is obtained by Lemmas \ref{k3}, \ref{k2}, and \ref{k1}.
\end{p a}
\section{ $2$-capability of some $ p$-groups } 
This section is devoted  to determine all $ 2 $-capable  finite $ 2 $-generator $ p$-groups of class two for odd prime $ p.$

The following results are useful for determining the $2$-capability of groups.
\begin{lem} \cite[Theorem 4.4]{MK} and \cite[Lemma 2.1$(vii)$]{BE}\label{normal}
  Let $ N $ be a normal subgroup of a group $ G $ contained in  $Z_2( G). $ Then 
  \begin{itemize}
  \item[$(i)$]$ N  \subseteq Z_2^{ *} (G) $ if and only if the natural map $ \mathcal{M}^{(2)}(G) \longrightarrow  \mathcal{M}^{(2)} ( G / N ) $ is a monomorphism.
 \item[$(ii)$]The sequence $ \mathcal{M}^{(2)}(G) \longrightarrow  \mathcal{M}^{(2)} ( G / N )\longrightarrow N\cap \gamma_3(G)\longrightarrow 1 $ is exact.
  \end{itemize}
    \end{lem}
An immediate result of Lemma \ref{normal} is as follows:
    \begin{cor}\label{normal2}
   Let $G$ be a group of nilpotency class two and  $ N $ be a normal subgroup of $ G. $  Then  $ N  \subseteq Z_2^{ *} (G) $ if and only if $ \mathcal{M}^{(2)}(G) \cong  \mathcal{M}^{(2)} ( G / N ).$
    \end{cor}
    
Note that,  the $2$-capability implies  that the capability. For this,  we just need to discuss the $2$-capability of  groups in Corollary \ref{2cap1}(a).
  \begin{thm}\label{ck3}
Let $G$ be a $p$-group such that $ G\cong G_1 $ or $ G\cong G_2 $ with the presentations as in Corollary \ref{2cap1}. Then $G$ is $2$-capable.
\end{thm}
\begin{proof}
Assume that $ G\cong G_1. $ From \cite[Theorem 1.3]{BE},  $ G_1/G_1'$ is  $2$-capable, and so $ Z_{2}^{ * } (G_1) \subseteq G_1'=\langle [a,b]\rangle.$ Let $ x $ be a nontrivial element  of  $ Z_{2}^{ * } (G_1) .$ If 
$ \langle x\rangle=\langle [a,b]^r\rangle$  with  $gcd( p, r)=1,$ then $ \langle x\rangle=G'. $ Using \cite[Theorem  2.3]{NP} and Lemma \ref{k3}, $ \mathcal{M}^{(2)} ( G_1/G_1' ) \cong \mathbb{Z}_{ p^{ \alpha}} ^{ (2) }$ and
 $  \mathcal{M}^{ (2) } (G_1) \cong \mathbb{Z}_{p^{ \alpha }}^{ (5) }.$ Therefore,  we will have a contradiction  by Lemma \ref{normal}$ (i).$ 
 Now, let  $gcd( p , r)\neq 1.$ Then $ r=p^s t$ for some $ s $ with $ 1\leq s< \alpha $ and $gcd( p , t)= 1.$  Hence  $ \langle x\rangle= \langle [a,b]^{p^s}\rangle.$  We denote the image of $ y\in G_1 $ in  $  G_1/\langle x\rangle$ by $\overline{y}.$   Thus
\begin{center}
  $ G_1 /\langle x\rangle = \langle \overline{a}, \overline{ b} \mid \overline{a}^{ p^{ \alpha }} = \overline{b}^{ p^{ \alpha }} = [ \overline{a} , \overline{b} ] ^{ p^{s}} = [ \overline{a} , \overline{b} , \overline{a} ] = [ \overline{a} , \overline{b} , \overline{b} ] =1,\alpha>s\rangle\cong G_2.  $
\end{center} 
By Lemmas \ref{k3} and \ref{k2},  $ \vert \mathcal{M}^{ (2) } (G_1) \vert >  \vert \mathcal{M}^{ (2) } (G_1/\langle x\rangle)  \vert,$  which is a contradiction,  using Lemma \ref{normal}$ (i).$  So, $ Z_{2}^{ * } (G_1)=1,$  and the result follows in this case. Now, let  $ G\cong G_2. $  By a similar way, we get $ Z_{2}^{ * } (G_2)=1,$ as desirable.
  \end{proof} 
  The following lemma is an essential key in the proof of Theorem \ref{ck11}.
  \begin{lem}\label{ck1}
Let $G$ be the group $G_3$ presented in Corollary \ref{2cap1} $(iii).$ Assume that   $1\neq d = a^{ i' p^{ \beta+k_1 }}b ^{ j' p^{ \beta+k_2 }}  $ and $ gcd( p , i ' ) = gcd(p , j ' ) = 1, $ where $ k_1 $ and $ k_2 $ are integers. Then the following results hold:
\begin{itemize}
\item[$  (i)$]$ k_1=k_2=k, $  $ d = a^{ i' p^{ \beta+k }}b ^{ j' p^{ \beta+k}},  $ and  $k<\alpha-\beta.$
\item[$  (ii)$]$|d|=   p^{ \alpha-\beta-k}.$  
\item[$  (iii)$]$ G/  \langle d \rangle\cong \langle a_{ 1} , b_{ 1} \ \vert 
\ a_{ 1} ^{ p^{ \alpha }} = b_{ 1} ^{ p^{ \alpha }}  = [ a_{ 1} , b_{ 1} ] ^{ p^{ \gamma } } = [ a_1, b_1, b_1 ]= [ a_1, b_1 , a_1] = 1 , $  $b_{ 1} ^{ p^{ \beta }} = [ a_{ 1} , b_{ 1} ] ^{ p^{ \sigma }} , a_{1} ^{p^{ \beta + k }} = b_{1} ^{ -  p^{ \beta + k }},\beta + k< \alpha \rangle .$
\item[$  (iv)$] $
 \mathcal{M}^{(2)}(G / \langle d \rangle ) \cong 
                                     \mathbb{Z}_{p^{ \beta +k }} \oplus \mathbb{Z}_{p^{ \beta }} \oplus \mathbb{Z}_{p^{ \sigma }}^{ ( 3) }.$
\end{itemize}
\end{lem}
\begin{proof}
Consider the image of $ y\in G $ in  $  G/ \langle d \rangle$ by $ \tilde{y}. $ As a result,
\begin{align*}
G/  \langle d \rangle= & \langle \tilde{a} , \tilde{b} \mid
 \tilde{a} ^{ p^{ \alpha }} = [ \tilde{a} , \tilde{b} ] ^{ p^{ \gamma } } = 1 , \tilde{b} ^{ p^{ \beta }} = [ \tilde{a} , \tilde{b} ] ^{ p^{ \sigma }} , 
 \tilde{a} ^{ i'p^{ \beta + k_1 }} = \tilde{b} ^{ - j' p^{ \beta + k_2 }}= [ \tilde{a} , \tilde{b} ] ^{-j' p^{ \sigma +k_2 }},
 \\& [ \tilde{a}, \tilde{b}, \tilde{b} ]= [ \tilde{a}, \tilde{b} , \tilde{a}] = 1 \rangle.
\end{align*}
Since $ \alpha-\beta=\gamma-\sigma $ and $ \vert \tilde{a}^{ i'p^{ \beta + k _{1} }} \vert = \vert \tilde{b} ^{ j'p^{ \beta + k _{2} }} \vert,$ we obtain  $  k_1=k_2. $ Put $ k_{1} = k_{2} = k.$ Thus
\[1\neq d = a^{ i' p^{ \beta+k }}b ^{ j' p^{ \beta+k}}  ~\text{ and }~ gcd( p , i ' ) = gcd(p , j ' ) = 1.\]
If $ \beta + k \geq \alpha, $ then $G/  \langle d \rangle\cong G,$ and so $d=1,$  which is  a contradiction. Thus  $ \beta + k< \alpha .$ The case $ (i) $ is obtained.
Now, we want to compute the order of $ d. $ Using \cite[Section 5.1 p.28]{MM} ,  $ Z( G) = \langle a^{ p^{ \gamma } } , [ a , b ] , b^{ p^{ \gamma } } \rangle,$ and hence, $\langle a^{ i' p^{ \beta+k }},b ^{ j' p^{ \beta+k}}\rangle \subseteq Z( G)$. 
Since $ \gamma-\sigma=\alpha -\beta>k $ and  $ |b|=  |a|= p^{ \alpha },$ we get $|a^{ i' p^{ \beta+k }}|=p^{\alpha-\beta-k  }=|b^{ j' p^{ \beta+k }}|= |d|=                                    p^{ \alpha-\beta-k}\neq 1.$
The case $( ii )$ is obtained.
Without loss of generality, taking $ i ' = j ' = 1, $ using \cite[Proposition 3.1]{AMM}, we have 
\begin{equation*}
\begin{split}
 G/  \langle d \rangle\cong H = \langle a_{ 1} , b_{ 1} \ \vert 
&\ a_{ 1} ^{ p^{ \alpha }} = b_{ 1} ^{ p^{ \alpha }}  = [ a_{ 1} , b_{ 1} ] ^{ p^{ \gamma } } = 1 , b_{ 1} ^{ p^{ \beta }} = [ a_{ 1} , b_{ 1} ] ^{ p^{ \sigma }} , a_{1} ^{p^{ \beta + k }} = b_{1} ^{ -  p^{ \beta + k }},\\
&\quad [ a_1, b_1, b_1 ]= [ a_1, b_1 , a_1] = 1,\beta + k< \alpha \rangle .
\end{split}
\end{equation*}
By  \cite[Proposition 1.1]{bey}, we get $  a_1 ^{ p^{ \beta + k}} \in Z^{*} ( H ) \subseteq Z_{2}^{ *} ( H )$.
 We denote the image of $ y\in H$ in  $  H/\langle a_1^{ p^{ \beta + k }} \rangle$ by $ \overline{y}. $ As a result, 
\begin{equation*}
\begin{split}
 H_1 = H / \langle a_1^{ p^{ \beta + k }} \rangle =   \langle \overline{a_1} , \overline{b_1 } \ \vert 
& \ \overline{a_1 } ^{ p^{ \beta + k }} =   [ \overline{ a_1} , \overline{ b_1 } ] ^{ p^{ \sigma + k }} = 1 , \overline{ b_1} ^{ p^{ \beta }} = \overline{[a_1,b_1]}^{ p^{ \sigma }},\\
&\quad [ \overline{a_1}, \overline{b_1}, \overline{b_1} ]= [ \overline{a_1}, \overline{b_1} , \overline{a_1}] = 1 \rangle .
\end{split}
\end{equation*}
Now, Lemma \ref{k1} and Corollary \ref{normal2} imply that 
\[
 \mathcal{M}^{(2)}(G / \langle d \rangle ) \cong  \mathcal{M}^{(2)} ( H) \cong \mathcal{M}^{(2)} ( H_1)\cong 
                                     \mathbb{Z}_{p^{ \beta +k }} \oplus \mathbb{Z}_{p^{ \beta }} \oplus \mathbb{Z}_{p^{ \sigma }}^{ ( 3) },\] and the case $ (iv) $ is proved.
\end{proof}
\begin{thm}\label{ck11}
Let $G$ be the group $G_3$ presented in Corollary \ref{2cap1} $(iii).$ Then $G$ is $2$-capable.
\end{thm}
\begin{proof}
Put $ M=  \langle a^{ p^{ \beta }} ,  b^{ p^{ \beta } }\rangle.$ Using \cite[Section 5.1 p.28]{MM},  $ Z( G_3) = \langle a^{ p^{ \gamma } } , [ a , b ] , b^{ p^{ \gamma } } \rangle,$ and so $ M \unlhd G_3$.  Consider the image of $ y\in G_3 $ in  $  G_3/M$ by $ \tilde{y}. $ As a result,
\begin{center}
$ G_3/ M = \langle \tilde{a} , \tilde{b} \ \vert \ \tilde{a}^{ p^{ \beta }} = \tilde{b}^{ p^{ \beta }} = [ \tilde{a} , \tilde{b} ] ^{ p^{ \sigma }} = [ \tilde{a} , \tilde{ b}, \tilde{ a} ] =   [ \tilde{a} , \tilde{ b}, \tilde{ b} ] = 1,\beta > \sigma \rangle \cong G_2.$
\end{center} 
From Theorem \ref{ck3}, $ G_3/ M $ is $2$-capable and hence $ Z_{2}^{ * } (G_3) \subseteq M.$ We claim that $  Z_{2} ^{ *} ( G_3)=1. $ By the way of contradiction, assume that $ 1 \neq d \in Z_{2} ^{ *} ( G_3)  $ is an arbitrary element. Then $ d = a^{ i p^{ \beta }}b ^{ j p^{ \beta }}  $ such that $ i $  and $ j  $ are integers. Suppose that $ i = i ' p^{ k_{ 1} } , j = j ' p^{ k_{ 2} },$ and $ gcd( p , i ' ) = gcd(p , j ' ) = 1, $ where $ k_1 $ and $ k_2 $ are integers. 
Lemmas \ref{k1} and \ref{ck1} imply  that
$
 \mathcal{M}^{(2)}(G_3 / \langle d \rangle ) \cong 
                                     \mathbb{Z}_{p^{ \beta +k }} \oplus \mathbb{Z}_{p^{ \beta }} \oplus \mathbb{Z}_{p^{ \sigma }}^{ ( 3) } $
and
$
 \mathcal{M}^{(2)}(G_3 ) \cong 
                                     \mathbb{Z}_{p^{ \alpha }} \oplus \mathbb{Z}_{p^{ \beta }} \oplus \mathbb{Z}_{p^{ \sigma }}^{ ( 3) }.$
   Hence, $ \vert \mathcal{M}^{(2)} (G_3) \vert > \vert \mathcal{M}^{(2)} (G_3 / \langle d \rangle ) \vert ,$  and we will have  $d\not\in Z_{2}^{ *} ( G),$ by Lemma  \ref{normal}$ (i).$  This contradiction completes the proof. 
\end{proof}
\begin{p b}
Let $G$ be $2$-capable. Then $G$ is capable, and so Corollary \ref{2cap1}(a) implies that $G$  is isomorphic to one of the groups $G_1, G_2,$ or $G_3.$ The converse holds by Theorems \ref{ck3} and \ref{ck11}.
\end{p b}
\begin{p e}
If $ G $ is  capable, then  $ Z^{ *} ( G) = Z_{2}^{ *} ( G) =1,$ using Theorem B. Otherwise, since $ G/Z^{ *} ( G) $ is nilpotent of  class at most $ 2, $ Corollary D and \cite[Theorem 1.3]{BE} imply that $ G/Z^{ *} ( G) $  is $2$-capable, and so $  Z_{2}^{ *} ( G)/Z^{ *} ( G)\subseteq Z_{2}^{ *} ( G/Z^{ *} ( G))=1.$ Hence, $ Z_{2} ^{ *} ( G)\subseteq Z^{ *} ( G) $  as required.
\end{p e}
\section{$2 $-Nilpotent multipliers of some noncapable $ p $-groups }
In this section, we intend to determine  the structures of $2 $-nilpotent multipliers for noncapable finite $2$-generator $ p $-groups of class two.

 A criteria for detecting a capable group is the notion of the exterior center. The exterior center of a group $G,$ $ Z^{\wedge}(G), $ is defined in \cite{el} as follows:
\[Z^{\wedge}(G)=\{g\in G\mid  g \wedge h = 1_{G\wedge G}~\text{for all}~ h\in G\},\]
where $ \wedge $ denotes the operator of the nonabelian exterior square. It is shown \cite{el} that 
$Z^*(G)=Z^{\wedge}(G).$ It implies that $ G $ is capable if and only if $Z^{\wedge}(G)=1.  $
\begin{lem}\label{nb10}
Let $G = G_p(\alpha,\beta,\gamma;\rho,\sigma)$  be a $2$-generator $p$-group of class two  presented as in Theorem \ref{pres}. If $ k >\mathrm{max} (\rho, \sigma), $ then
 $$ a^{p^{k}}\wedge b=a \wedge b^{p^{k}} = (a\wedge b)^{\frac{p^{k}(p^{k}+1)}{2}} $$

 \end{lem}
 \begin{proof}We claim that $a^{p^{k}}\wedge b=(a\wedge b)^{\frac{p^{k}(p^{k}+1)}{2}}.$ For this, we have
  \begin{align*}
   a^{p^{k}}\wedge b &=\prod_{i=0}^{p^{k}-1} (^{a^i}(a\wedge b))=\prod_{i=0}^{p^{k}-1} (a\wedge ^{a^i}b)\\&=\prod_{i=0}^{p^{k}-1} (a\wedge [a^i,b]b)=\prod_{i=0}^{p^{k}-1} (a\wedge [a^i,b])  (a\wedge b)\\&=(a\wedge b)^{1+\ldots+p^{k}}\prod_{i=1}^{p^{k}-1} (a\wedge [a,b])^i \\&=(a\wedge b)^{\frac{p^{k}(p^{k}+1)}{2}}\prod_{i=1}^{p^{k}-1} (a\wedge [a,b])^i\\&
  =(a\wedge b)^{\frac{p^{k}(p^{k}+1)}{2}}(a\wedge [a,b])^{1+\ldots+p^{k}-1}\\&=(a\wedge b)^{\frac{p^{k}(p^{k}+1)}{2}}
(a\wedge [a,b])^{\frac{p^{k}(p^{k}-1)}{2}} \\&=(a\wedge b)^{\frac{p^{k}(p^{k}+1)}{2}}(a\wedge [a,b]^{\frac{p^{k}(p^{k}-1)}{2}})\\&=
(a\wedge b)^{\frac{p^{k}(p^{k}+1)}{2}}\big{(} a\wedge ([a,b]^{p^{\rho}})^{\frac{p^{k-\rho}(p^{k}-1)}{2}}\big{)}\\&=
(a\wedge b)^{\frac{p^{k}(p^{k}+1)}{2}} \big{(} a\wedge (a^{p^{\alpha}})^{\frac{p^{k-\rho}(p^{k}-1)}{2}}\big{)}=
(a\wedge b)^{\frac{p^{k}(p^{k}+1)}{2}},
  \end{align*}
   and so\begin{equation}\label{8}
   a^{p^{k}}\wedge b=(a\wedge b)^{\frac{p^{k}(p^{k}+1)}{2}}.
\end{equation} 
  Similarly, 
  \begin{equation}\label{7}
   a\wedge b^{p^{k}}=(b^{p^{k}}\wedge a)^{-1}=(b\wedge a)^{-\frac{p^{k}(p^{k}+1)}{2}},
  \end{equation}
  as desirable.
  
 \end{proof}
 
Note that, in  Lemma \ref{nb10}, the condition $ k >\mathrm{max} (\rho, \sigma)$  may be replaced by $k\geq \mathrm{max} (\rho, \sigma),$ whenever  $p$ is an odd prime number. \par

Using Theorem  \ref{pres} and Lemma \ref{mcap}, all noncapable finite $2$-
generator $p$-groups of class two are described as follows:
\begin{cor}\label{cas}
Let $ G $ be a  finite $ 2 $-generator $ p $-group of class two  with the presentation as in Theorem  \ref{pres}. Then the following results hold:\\If  $p$ is an odd prime number, then $G$ is noncapable if and    only if $G$  is isomorphic to exactly one of the following groups.
\item[$(1)$]  $K_1=\langle a, b \ \vert \ [ a , b ] ^{ p^{ \gamma }} =[ a, b, b ]= [ a, b , a] = 1=b ^{p^{ \beta } }, a ^{ p^{ \alpha }} =[ a , b ] ^{ p^{ \rho }}\rangle,$ where $\alpha>\beta \geq\gamma$  and $0\leq \rho < \gamma.$ 
\item[$(2)$] $K_2=\langle a, b \mid a ^{ p^{ \alpha }} =b ^{p^{ \beta } } =[ a , b ] ^{ p^{ \gamma }} =[ a, b, b ]= [ a, b , a] = 1 \rangle,$ where $\alpha>\beta \geq \gamma.$  
\item[$(3)$] $K_3=\langle a, b \ \vert \ \ a ^{ p^{ \alpha }}=[ a , b ] ^{ p^{ \gamma }} =[ a, b, b ]= [ a, b , a] = 1 ,  b ^{p^{ \beta } } =[ a , b ] ^{ p^{ \sigma }} \rangle,$ where $\alpha>\beta\geq \gamma$  and $0 \leq\sigma<\sigma+\alpha-\beta < \gamma.$      
\item[$(4)$] $K_4=\langle a, b \ \vert \ [ a , b ] ^{ p^{ \gamma }} =[ a, b, b ]= [ a, b , a] = 1,a ^{p^{ \alpha } }=[ a , b ] ^{ p^{ \rho }},\ b ^{ p^{ \beta }} =[ a , b ] ^{ p^{ \sigma }}\rangle,  $ where $\alpha>\beta\geq\gamma$  and $0\leq \sigma<\rho< \mathrm{min} (\gamma, \sigma+\alpha-\beta).$
\item[$(5)$] $K_5=\langle a, b \ \vert \ [ a , b ] ^{ p^{ \gamma }} =[ a, b, b ]= [ a, b , a] = 1=b ^{p^{ \alpha} },\ a ^{ p^{ \alpha }} =[ a , b ] ^{ p^{ \rho }} \rangle,$ where $\alpha>\gamma>\rho.$
\\
 If  $p=2,$ then  $G$ is noncapable if and only if $G$  is isomorphic to exactly one of the following groups.
\item[$(6)$] $K_6=\langle a, b \ \vert \ [ a , b ] ^{ 2^{ \gamma }} =[ a, b, b ]= [ a, b , a] = 1,\ a ^{ 2^{ \alpha }} =[ a , b ] ^{ 2^{ \rho }},  b ^{2^{ \beta } } =1  \rangle,$ where $\alpha>\beta \geq\gamma$  and $0\leq \rho<\gamma.$ 
\item[$(7)$]$K_7=\langle a, b \ \vert \ [ a , b ] ^{ 2^{ \gamma }} =[ a, b, b ]= [ a, b , a] = 1,\ a ^{ 2^{ \alpha }} =1=  b ^{2^{ \beta } } \rangle,$ where $\alpha>\beta >\gamma.$  
\item[$(8)$]$K_8=\langle a, b \ \vert \ [ a , b ] ^{ 2^{ \beta  }} =[ a, b, b ]= [ a, b , a] = 1,\ a ^{ 2^{ \alpha }} =1= b ^{2^{ \beta } }  \rangle,$ where $\alpha>\beta+1. $  
\item[$(9)$]$K_9=\langle a, b \ \vert \ [ a , b ] ^{ 2^{ \gamma }} =[ a, b, b ]= [ a, b , a] = 1,\ a ^{ 2^{ \alpha }} =1  , b ^{2^{\beta} } =[ a , b ] ^{ 2^{ \sigma }}\rangle,$ where $0 \leq\sigma<\sigma+\alpha-\beta<\gamma.$
\item[$(10)$]$K_{10}=\langle a, b \ \vert \ [ a , b ] ^{ 2^{ \gamma }} =[ a, b, b ]= [ a, b , a] = 1,\ a ^{ 2^{ \alpha }} =[ a , b ] ^{ 2^{ \rho }}, b ^{2^{\beta } } =[ a , b ] ^{ 2^{ \sigma }}\rangle,$  where $0\leq \sigma<\rho< \mathrm{min} (\gamma, \sigma+\alpha-\beta).$
\item[$(11)$]$K_{11}=\langle a, b \ \vert \ [ a , b ] ^{ 2^{ \gamma }} =[ a, b, b ]= [ a, b , a] = 1,\ a ^{ 2^{ \alpha }} = [ a , b ] ^{ 2^{ \rho }}, b ^{2^{ \alpha } }=1 \rangle,$ where $\alpha>\gamma> \rho.$  
\item[$(12)$]$K_{12}=\langle a, b \ \vert \ [ a , b ] ^{ 2^{ \alpha }} =[ a, b, b ]= [ a, b , a] = 1,\ a ^{ 2^{ \alpha }} = [ a , b ] ^{ 2^{ \rho }}, b ^{2^{ \alpha } }=1 \rangle,$ where $\alpha> \rho.$ 
\item[$(13)$]$K_{13}=\langle a, b \ \vert \ [ a , b ] ^{ 2^{ \beta }} =[ a, b, b ]= [ a, b , a] = 1,\ a ^{ 2^{ \beta }} = [ a , b ] ^{ 2^{ \rho }}, b ^{2^{ \beta } }=1 \rangle,$ where $\beta-1> \rho.$
\item[$(14)$]$K_{14}=\langle a, b \ \vert \ [ a , b ] ^{ 2^{ \alpha }} =[ a, b, b ]= [ a, b , a] = 1,\ a ^{ 2^{ \alpha }} = [ a , b ] ^{ 2^{ \alpha-1 }}= b ^{2^{ \alpha } } \rangle.$  
 \end{cor}
The following result gives an element in the epicenter of some noncapable finite $2$-generator $p$-groups of class two. 
 \begin{thm}\label{nb}
 Let $ G= K_i$   be the $2$-generator $p$-group of class two  presented as in Corollary \ref{cas}. Then for  $i\neq 8,$ we have
 \begin{itemize}
\item[$(i)$] $b^{{p}^{\alpha}}\in Z_2^{*}(G),$ if  $ a^{p^{ \alpha}}=1;$ 
\item[$(ii)$]$a^{{p}^{\beta}}\in Z_2^{*}(G),$ if  $ b^{p^{ \beta}}=1;$
\item[$(iii)$] $a^{{p}^{\alpha}}\in  Z_2^{*}(G),$ if  $\sigma<\rho.$
\end{itemize}
 \end{thm}
 \begin{proof}
Let $G = K_i.$ Then $Z^{*}(G)=Z_2^{*}(G),$ by Theorem E. It is obvious that $ x\in Z^{\wedge}(G)=Z^{*}(G)$ if and only if $x\wedge a=1=x\wedge b.$ Therefore, for each case, it is enough to show that three elements $b^{{p}^{\alpha}}\wedge a,$  $ a^{{p}^{\beta}}\wedge b,$ and $a^{{p}^{\alpha}}\wedge b$ are trivial, respectively. Now, it is easy to achieve the desired result, by using a similar method as stated in Lemma \ref{nb10}.
 \end{proof}
 
 \begin{thm}\label{uu}
Assume that  $ G $ is a  noncapable $ 2$-generator $p$-group of class two  given  in Corollary \ref{cas}$\ (1)$-$ (5)$. Then the structure of $2$- nilpotent multiplier of $G$ is as follows:
\begin{itemize}
\item[$(1)$]Let $ G\cong K_1. $  Then 
$\mathcal{M}^{(2)} (G) \cong   \mathbb{Z}_{ p^{ \beta}}^{(2)}\oplus \mathbb{Z}_{ p^{\rho}}^{ (3)}.$
\item[$(2)$]Let $ G\cong K_2.$ Then 
$ \mathcal{M}^{(2)}(G) \cong  \mathbb{Z}_{ p^{ \beta}}^{ (2)}\oplus \mathbb{Z}_{ p^{ \gamma}}^{ (3) } $
\item[$(3)$]Let $ G\cong K_3.$ Then 
$
 \mathcal{M}^{(2)}(G) \cong  \mathbb{Z}_{ p^{ \alpha}} \oplus  \mathbb{Z}_{ p^{ \beta}}\oplus \mathbb{Z}_{ p^{ \sigma}}^{ (3) }.$
\item[$(4)$]Let $ G\cong K_4.$ Then 
$
 \mathcal{M}^{(2)}(G) \cong \mathbb{Z} _{ p^{  \rho-\sigma+\beta}  }\oplus\mathbb{Z}_{ p^{ \beta}}\oplus \mathbb{Z}_{ p^{ \sigma} }^{ (3) }.$
 \item[$(5)$] Let $ G\cong K_5.$ Then 
$\mathcal{M}^{(2)} (G) \cong \mathbb{Z}_{ p^{ \alpha}}^{ (2)}\oplus \mathbb{Z}_{ p^{\rho}}^{ (3) }.$  
\end{itemize}
\end{thm}
\begin{proof}
\begin{itemize}
\item[$(1)$] Let $ G\cong K_1. $ Then by Theorem \ref{nb}$ (ii),\   a^{p^{\beta}} \in Z_2^{*}(G). $ Since $G/\langle a^{p^{\beta}}\rangle \cong G_2,$  we will have the result by Theorem A and  Corollary \ref{normal2}.
  \item[$(2)$] Let $ G\cong K_2.$ It is not difficult  to see that $G/\langle a^{p^{\beta}}\rangle \cong G_1,$ when $ \beta=\gamma,$ and otherwise $G/\langle a^{p^{\beta}}\rangle \cong G_2.$ Now, the result is in hand by Theorem \ref{nb}$(ii),$ Theorem A, and Corollary \ref{normal2}.
  \item[$(3)$] Let $ G\cong K_3.$ Theorem \ref{nb}$(i)$ implies that $b^{p^{\alpha}}\in Z_2^{*}(G),$ and hence 
  $\mathcal{M}^{(2)} (G) \cong \mathcal{M}^{(2)} (G/\langle b^{p^{\alpha}}\rangle),$ by  Corollary \ref{normal2},
  in which
     \[ G/ \langle b^{p^{\alpha}}\rangle \cong \langle a_1, b_1 \mid   a_1 ^{ p^{ \alpha }}= b_1 ^{p^{ \alpha} }=[ a_1 , b_1 ] ^{ p^{ \alpha-\beta+\sigma }} =1, b_1 ^{p^{ \beta } } =[ a_1 , b_1 ] ^{ p^{ \sigma }},\]\[[ a_1, b_1, b_1 ]= [ a_1, b_1 , a_1] = 1, 0\leq \sigma <\sigma+\alpha-\beta,\beta<\alpha \rangle.\] Since  $ \alpha -\beta=\alpha-\beta+\sigma-\sigma,$ we get $ G/ \langle b^{p^{\alpha}}\rangle $ is capable. So,  $G/\langle b^{p^{\alpha}}\rangle \cong G_3.$  The result is concluded that by Theorem A. 
 \item[$(4)$] Let $ G\cong K_4.$ From Theorem \ref{nb}$ (iii),$  we have $a^{p^{\alpha}}\in Z_2^{*}(G).$ Since,
\[ G/ \langle a^{p^{\alpha}}\rangle \cong H= \langle a_1, b_1 \mid   a_1 ^{ p^{ \alpha }}=b_1 ^{ p^{ \rho-\sigma+\beta }}=[ a_1 , b_1 ] ^{ p^{ \rho }} =1, b_1 ^{p^{ \beta } } =[ a_1 , b_1 ] ^{ p^{ \sigma }},\]\[[ a_1, b_1, b_1 ]= [ a_1, b_1 , a_1] = 1,0\leq \sigma<\rho< \sigma+\alpha-\beta, \beta<\alpha \rangle,\]
and $ \rho< \sigma+\alpha-\beta, $ we have $ G/ \langle a^{p^{\alpha}}\rangle $ is noncapable.  Theorem \ref{nb}$ (ii)$ follows that $ a_1^{ p^{ \rho-\sigma+\beta }} \in Z_2^*(H).$ So, $\mathcal{M}^{(2)} (G)\cong \mathcal{M}^{(2)} (H) \cong \mathcal{M}^{(2)} (H/ \langle  a_1^{ p^{ \rho-\sigma+\beta }}\rangle), $ by  Corollary \ref{normal2}. Now, one can obtain the result by Theorem A, because of
 $ H/ \langle  a_1^{ p^{ \rho-\sigma+\beta }}\rangle \cong G_3.$  
  \item[$(5)$] Let $ G\cong K_5.$  Then $a^{p^{\alpha}} \in Z^*(G),$ by Theorem \ref{nb}$ (ii)$. Therefore $\mathcal{M}^{(2)} (G) \cong \mathcal{M}^{(2)} (G/\langle a^{p^{\alpha}}\rangle)\cong  \mathcal{M}^{(2)} (G_2),$  as required.
\end{itemize}
\end{proof}
 
The next result gives the structure of $2$- nilpotent multiplier of other two generator groups. 
 \begin{thm}
Let $ G $ be a  noncapable $ 2$-generator $2$-group of  class two   given  in Corollary  \ref{cas}$\ (6)$-$ (14)$.
\begin{itemize}
\item[(1)]Let $ G\cong K_6. $ Then
$\mathcal{M}^{(2)} (G) \cong  \mathbb{Z}_{ 2^{ \beta}}^{ (2)}\oplus \mathbb{Z}_{ 2^{\rho}}^{ (3)}.$
\item[(2)]Let $ G\cong K_7. $ Then
 $\mathcal{M}^{(2)}(G) \cong  \mathbb{Z}_{ 2^{ \beta}}^{ (2)}\oplus \mathbb{Z}_{ 2^{\gamma}}^{ (3)}.$
\item[(3)]Let $ G\cong K_8. $ Then
$\mathcal{M}^{(2)} (G) \cong  \mathbb{Z}_{ 2^{ \beta-1}}^{ (3)}\oplus \mathbb{Z}_{ 2^{\beta}}\oplus \mathbb{Z}_{ 2^{\beta+1}}.$
\item[(4)]Let $ G\cong K_9. $ Then
$
 \mathcal{M}^{(2)}(G) \cong \mathbb{Z}_{ 2^{ \alpha}} \oplus  \mathbb{Z}_{ 2^{ \beta}}\oplus \mathbb{Z}_{ 2^{ \sigma}}^{ (3) }.$
\item[(5)]Let $ G\cong K_{10}. $ Then
$
 \mathcal{M}^{(2)}(G) \cong \mathbb{Z} _{ 2^{  \rho-\sigma+\beta}  }\oplus\mathbb{Z}_{ 2^{ \beta}}\oplus \mathbb{Z}_{ 2^{ \sigma} }^{ (3) }.$
\item[(6)]Let $ G\cong K_{11}. $ Then
$\mathcal{M}^{(2)} (G) \cong  \mathbb{Z}_{ 2^{ \alpha}}^{ (2)}\oplus \mathbb{Z}_{ 2^{\rho}}^{ (3)}.$\item[(7)]Let $ G\cong K_{12}. $ Then
$\mathcal{M}^{(2)} (G) \cong  \mathbb{Z}_{ 2^{ \alpha}}^{ (2)}\oplus \mathbb{Z}_{ 2^{\rho}}^{ (3)}.$\item[(8)]Let $ G\cong K_{13}. $ Then
$\mathcal{M}^{(2)} (G) \cong  \mathbb{Z}_{ 2^{ \beta}}^{ (2)}\oplus \mathbb{Z}_{ 2^{\rho}}^{ (3)}.$\item[(9)]Let $ G\cong K_{14}. $ Then
$\mathcal{M}^{(2)} (G) \cong  \mathbb{Z}_{ 2^{ \alpha}}^{ (2)}\oplus \mathbb{Z}_{ 2^{\alpha-1}}^{ (3)}.$
\end{itemize}
\end{thm}
\begin{proof} 
\begin{itemize}
\item[$(1)$]Let $ G\cong K_i,$  for $i=6,7,11,12,13,14.$ By a similar way used in Theorem \ref{uu},  we can observe that  $ \mathcal{M}^{(2)} (G) \cong\mathcal{M}^{(2)}(G_6), $  and the result follows by  \cite[Theorem 3.3]{FKJ}.  Let $G\cong K_8.$ Since $G/\langle a^{2^{\beta}}\rangle$ is $2$-capable, we will have $1\neq Z_2^{*}(G)\subseteq \langle a^{2^{\beta}}\rangle. $  Hence, $Z_2^{*}(G)=\langle a^{2^{\beta+q}}\rangle, $  and therefore  $a^{2^{\alpha-1}}\in Z_2^{*}(G),$ because of $\alpha>\beta+1.$ Let $H_0=G$ and $H_i=H_{i-1}/\langle a^{2^{\alpha-i}}\rangle $ for all $ i $ with $1\leq i\leq \alpha-\beta.$ It is easy to see that $ \mathcal{M}^{(2)} (H_{i-1})\cong  \mathcal{M}^{(2)} (H_{i}),$ whenever $\alpha -i=\beta.$ Hence, if $\alpha -1=\beta,$ then $ \mathcal{M}^{(2)} (G)\cong  \mathcal{M}^{(2)} (G/\langle a^{2^{\beta}}\rangle) \cong  \mathcal{M}^{(2)} (G_4).$ Since $\alpha>\beta,$ then there exists an element $j$ such that $\alpha -j=\beta.$ We increase $i$ from $1$ up to $j$  and calculate $H_i$' s.  Now, one can check that $ \mathcal{M}^{(2)} (G)\cong\mathcal{M}^{(2)} (H_{j})\cong\mathcal{M}^{(2)} (H_{j}/\langle a^{2^{\alpha-j}} \rangle)\cong  \mathcal{M}^{(2)} (G/\langle a^{2^{\beta}}\rangle),$ and obtain $ \mathcal{M}^{(2)} (H_i)\cong  \mathcal{M}^{(2)} (G_4) $ for all $ i $ with $0\leq i\leq \alpha-\beta.$ Hence, the result is in hand by \cite[Theorem 3.3]{FKJ}. 
    Now, let $G\cong K_9.$ Then
\begin{center}
$G/\langle b^{2^{\alpha}}\rangle \cong  \langle a_1,b_1\mid [a_1,b_1]^{2^{\sigma+\alpha - \beta}}=a_1^{2^{\alpha}}=1, b_1^{2^{\beta}}=[a,b]^{2^{\sigma}},\beta>\sigma\rangle 
\cong G_7.$
\end{center} 
  Since
  $b^{2^{\alpha}}\in Z^*(G)=Z_2^*(G),$ by Theorem \ref{nb}, we will have  $  \mathcal{M}^{(2)} (G) \cong \mathcal{M}^{(2)} (G /\langle b^{2^{\alpha}} \rangle)\cong  \mathcal{M}^{(2)} (G_7),$  as desirable.  For $ G\cong K_{10},$ the result is attained by a  similar method to $K_9.$

\end{itemize}
\end{proof}

\bigskip
\section*{Acknowledgement} F. Johari  was  supported  by the postdoctoral grant  `` CAPES/PRINT-Edital $n^\circ 41/2017$, Process number: 88887.511112/2020-00,''  at the Federal University of Minas Gerais.

\hspace{1in}

\end{document}